\documentclass{amsart} % AMS math article
\usepackage{amsmath,amssymb,amsthm} % formulas, mathematical symbols, theorems, etc.
\usepackage{graphicx} % for graphics
\usepackage{subfig}
\usepackage[arrow, matrix, curve]{xy} % for commutative diagrams
\usepackage{xcolor} % colors
\usepackage{makecell} % for linebreak in tables
\usepackage{multirow} % for multiple rows in tables
\usepackage{enumerate}
\usepackage{booktabs} % for pretty tables
\usepackage{siunitx} % for decimal alignment
\usepackage{hyperref} % for links

\newcommand{\Z}{\mathbb{Z}}

\newcommand{\N}{\mathbb{N}}

\makeatletter
\newcommand{\Vast}{\bBigg@{2.5}} % or 4.3?   To make large symbol
\makeatother

\makeatletter
\newcommand{\smallbullet}{} % for safety
\DeclareRobustCommand\smallbullet{%
  \mathord{\mathpalette\smallbullet@{0.7}}
}
\newcommand{\smallbullet@}[2]{%
  \vcenter{\hbox{\scalebox{#2}{$\m@th#1\bullet$}}}%
}
\makeatother

\newtheoremstyle{thm}{}{}{\itshape}{}{\bfseries}{}{ }{} %Thereom style
\newtheoremstyle{definition}{}{}{}{}{\bfseries}{}{ }{} %Definition style

\theoremstyle{thm}
\newtheorem{Theorem}{Theorem}[section]
\newtheorem{thm}[Theorem]{Theorem}

\newtheorem*{Theorem-ohne}{Theorem}

\theoremstyle{definition}

\newtheorem{rem}[Theorem]{Remark}

\begingroup\expandafter\expandafter\expandafter\endgroup
\expandafter\ifx\csname pdfsuppresswarningpagegroup\endcsname\relax
\else
\fi

\definecolor{amaranth}{rgb}{0.9, 0.17, 0.31} %dark red
\definecolor{carrotorange}{rgb}{0.93, 0.57, 0.13} %orange
\definecolor{citrine}{rgb}{0.89, 0.82, 0.04} %dark yellow
\definecolor{dartmouthgreen}{rgb}{0.05, 0.5, 0.06} %green
\definecolor{ballblue}{rgb}{0.13, 0.67, 0.8} %blue
\definecolor{ceruleanblue}{rgb}{0.16, 0.32, 0.75} %deeper blue
\definecolor{amethyst}{rgb}{0.6, 0.4, 0.8} %purple
\definecolor{amber}{rgb}{1.0, 0.75, 0.0} %amber
\definecolor{burlywood}{rgb}{0.87, 0.72, 0.53} %beigebrown
\numberwithin{equation}{section}
\begin{document}

%%%%%%%%%%%%%%%%%%%%%%%%%%%%% Title and authors %%%%%%%%%%%%%%%%%%%%%%%%%%%%%%%%%%%%
\title{Knots not detected by any trace} 

\author{Kenneth L. Baker}
\address{Department of Mathematics, University of Miami, Coral Gables, FL 33146, USA}
\email{k.baker@math.miami.edu}

\author{Marc Kegel}
\address{Universidad de Sevilla, Dpto.\ de Álgebra,
Avda.\ Reina Mercedes s/n,
41012 Sevilla}
\email{mkegel@us.es, kegelmarc87@gmail.com}

\author{Kimihiko Motegi}
\address{Department of Mathematics, Nihon University, 3-25-40 Sakurajosui, Setagaya-ku, Tokyo 156--8550, Japan}
\email{motegi.kimihiko@nihon-u.ac.jp}

%%%%%%%%%%%%%%%%%%%%%%%%%%%%% Abstract %%%%%%%%%%%%%%%%%%%%%%%%%%%%%%%%%%%%

%\date{\today} % date on first page

\begin{abstract}
The first and last named authors have demonstrated the existence of knots for which every integral slope is non-characterizing. 
In this short note, we extend this result in two ways. 
There exists a knot that shares for every integer $n$ the same $n$-trace with infinitely many mutually distinct knots. 
Moreover, every knot is concordant to a knot that is not detected by any of its traces. 
\end{abstract}


\makeatletter
\@namedef{subjclassname@2020}{%
  \textup{2020} Mathematics Subject Classification}
\makeatother%For 2020

 % Mathematical subject classification

% 57K10 Knot theory
% 57K14 Knot polynomials
% 57K16 Finite-type and quantum invariants, topological quantum field theories (TQFT)
% 57K32 Hyperbolic 3-manifolds
% 57M12 Low-dimensional topology of special (e.g., branched) coverings
% 57R58 Floer homology
% 57R65 Surgery and handlebodies

\maketitle

\section{Introduction}
\label{intro}
The $n$-\textit{trace} $X_n(K)$ of a knot $K$ is the oriented, smooth $4$-manifold obtained by attaching a $2$-handle to the $4$--ball $D^4$ attached along $K$ with framing $n$. We say that $K$ is \textit{detected} by its $n$-trace, if whenever $X_n(K)$ is orientation-preservingly diffeomorphic to $X_n(K')$ for another knot $K'$, then $K'$ is isotopic to $K$.

Question 1.9 from~\cite{Baldwin_Sivek_L_spaces} asks 
if every knot is detected by at least one of its traces. 
In this short note, we answer this question in the negative, 
demonstrating that this is quite far from the case. 

\begin{thm}
\label{thm:main}
There exists a knot $K$ such that for every integer $n\in\Z$, there exist an infinite family of mutually distinct knots $(K^k_n)_{k\in \N}$ such that each $K^k_n$ is distinct from $K$ but its $n$-trace $X_{n}(K^k_n)$ is orientation-preservingly diffeomorphic to $X_n(K)$.
\end{thm}

The knots used to prove Theorem~\ref{thm:main} admit annulus presentations that induce ribbon immersions of once-punctured tori. Hence these knots have $4$-genus at most one.
In contrast, our next result demonstrates the existence of knots 
with arbitrarily large $4$-genus that are not detected by any of their traces.

\begin{thm}
\label{thm:concordance}
For any knot $J$, there is a knot $J'$ concordant to $J$ that is not detected by any of its traces. Furthermore, we have infinitely many distinct knots $J'$ satisfying this property.
\end{thm} 

\begin{rem}
    Actually, our proof shows that $J'$ is ribbon concordant to $J$.
\end{rem}

\section{Knots having, for all integers \texorpdfstring{$n$}{n}, the same \texorpdfstring{$n$}{n}-trace as infinitely many other knots.}

Let $A\cup D$ be a closed disk embedded in $S^3$ where $D$ is a disk embedded in its interior and $A$ is the annulus that is the complement of the interior of $D$.  Let $c$ be an unknotted circle with framing $-1$ that is disjoint from $A$ and transversally intersects $D$ once.

Take an embedded  band $b$ connecting the two components of $\partial A$ which intersects $\mathrm{int}\,A$ in arcs and does not intersect $c$ as shown in Figure~\ref{fig:annulus_presentation}.
Then banding $\partial A$ together along $b$ yields a knot $K$
which has an {\em annulus presentation} $(A, b, c)$. 
We call an annulus presentation {\em simple} if $b \cap D = \emptyset$; Figure~\ref{fig:annulus_presentation} gives a simple annulus presentation.

\begin{figure}
    \centering
    \includegraphics[width=0.4\linewidth]{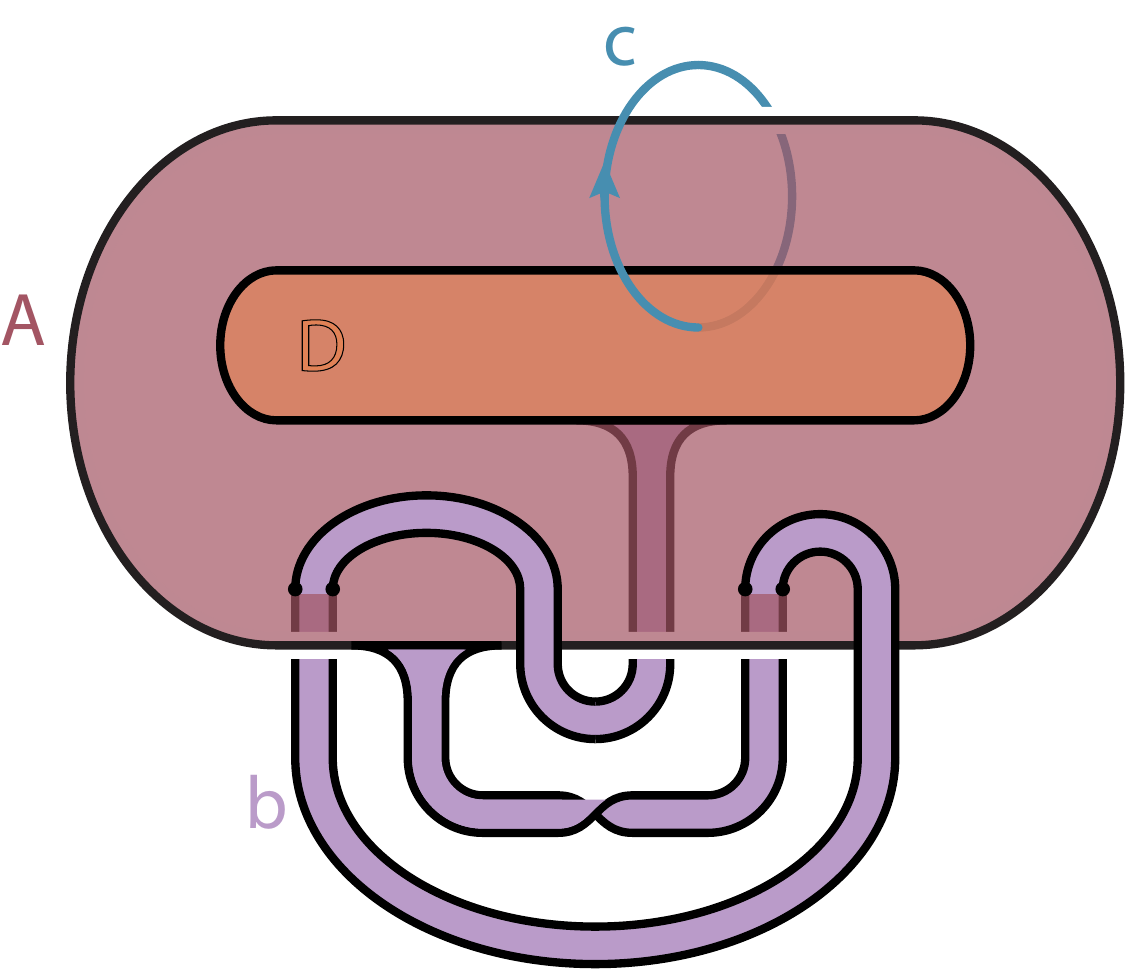}
    \caption{A simple annulus presentation.}
    \label{fig:annulus_presentation}
\end{figure}

\begin{proof}[Proof of Theorem~\ref{thm:main}]
This essentially follows from \cite{AJDLO_same_traces} by finding a knot that has two `simple annulus presentations' on Hopf bands of opposite handedness, that are both `good'.   
By \cite[Theorem 3.10]{AJDLO_same_traces}, a simple annulus presentation for a knot $K$ enables the construction of infinitely many knots with the same $n$-trace as $K$ for any integer $n$. 
By \cite[Theorem 3.13]{AJDLO_same_traces}, when a simple annulus presentation on a Hopf band with a positive twist is `good' and $n$ is positive, these infinitely many knots are distinguished by the degrees of their Alexander polynomials. 
What it means for an annulus presentation to be `good' is rather technical,
so we refer \cite[Definition 3.14]{AJDLO_same_traces} for the precise definition.
However, we illustrate the idea in Figure~\ref{fig:GAP8_9} and its caption.

Then by mirroring, a good annulus presentation on a Hopf band with a negative twist yields infinitely many distinct knots with the same $n$-trace as $\overline{K}$ for each negative integer $n$, where $\overline{K}$ denotes the mirror image of $K$.

This observation leads us to consider an amphicheiral knot $K$ with a good annulus presentation. 
If the $n$-trace $X_n(K)$ of an amphicheiral knot $K$ is orientation-preservingly diffeomorphic to the $n$-trace $X_n(K')$ of another knot $K'$, then it follows that $X_{-n}(K)\cong X_{-n}(\overline{K})\cong X_{-n}(\overline{K'})$. In particular, we deduce that an amphicheiral knot $K$ shares its $n$-trace with infinitely many distinct knots if and only if it shares its $(-n)$-trace with infinitely many distinct knots.

Let us consider the case where $n = 0$. 
If two knots share the same $0$-surgery, their Alexander polynomials agree~\cite{Rolfsen}. 
Nevertheless, for $n=0$, these infinitely many knots $K^k_0$ admit a surgery description on a single 3-component link $L$, see for example~\cite{Osoinach_annulus}. 
If $L$ is hyperbolic, then infinitely many of the knots $K^k_0$ with the same $0$-trace as $K$ will be hyperbolic and the volumes of $K^k_0$ converge strictly monotonically to the volume of $L$~\cite{volumes}. 
In particular, infinitely many of the $K_0^k$ can be distinguished by their hyperbolic volume~\cite{Osoinach_annulus}. 

Now let us exhibit an amphicheiral knot with good annulus presentations for which the associated link $L$ is hyperbolic, thereby proving the theorem.

The left column of Figure~\ref{fig:GAP8_9} shows a simple annulus presentation 
for the amphicheiral unknotting number 1 knot $8_9$, 
as obtained from \cite[Figure 9]{Abe_Tagami}. 

The first frame shows the annulus presentation with the push-offs of the two boundary components of the Hopf band, and the second shows it with an unknotting circle. The third picture shows, after the unknotting twist as in Figure~\ref{fig:annulus_presentation}, the now unknotted knot as a flat loop bounding an obvious disk with the signed intersections (for some choice of orientations) of the stretched unknotting circle resulting from pulling $\partial D$ along $b$ to the other component of $\partial A$. From those intersections, the last frame identifies the $(+-)$ and $(-+)$ arcs of the stretched unknotting circle which, along with the remaining arcs being only $(++)$ or $(--)$, certify these simple annulus presentations as `good'.
Finally, the verified functions in SnapPy~\cite{SnapPy} identify the complement of the 3-component link on the left as a hyperbolic manifold.
\end{proof}

\begin{figure}
    \centering
    \includegraphics[width=0.7\linewidth]{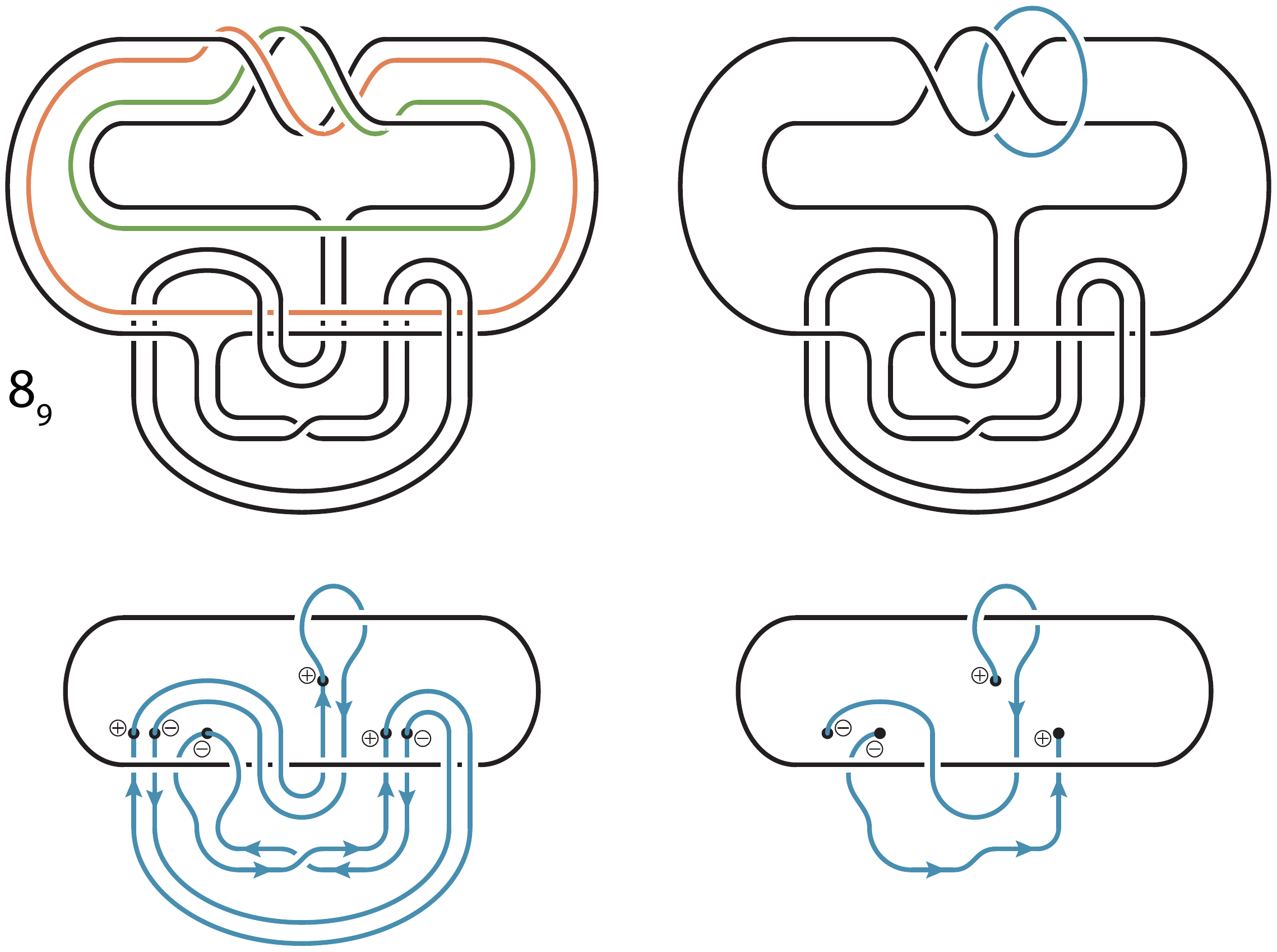}
    \caption{In the first frame, the knot $8_9$ is shown in a good annulus presentation along with push-offs of a positive Hopf band.  The second frame shows it with an unknotting circle. The third shows, after the unknotting twist (as in Figure~\ref{fig:annulus_presentation}), the now unknotted knot and the stretched unknotting circle. The last frame shows the $(+-)$ and $(-+)$ arcs that certify this annulus presentation as `good'.  }
    \label{fig:GAP8_9}
\end{figure}

\begin{rem}
    The knot $8_{17}$ with its simple annulus presentation given in \cite[Figure 9]{Abe_Tagami} also works.
\end{rem}

\begin{rem}
    Using an amphicheiral knot simplified the proof above. Alternatively, one could begin to build other examples by constructing a knot $K$ with unknotting crossings of opposite sign to produce two simple annulus presentations on Hopf bands of opposite sign.  
    Figure~\ref{fig:9_44} shows a presentation of the knot $9_{44}$ that indicates such a family of Montesinos knots.  
    As above, \cite[Theorem 3.13]{AJDLO_same_traces} will then yield, for each integer $n$, infinitely many knots with the same $n$-trace as $K$. To distinguish these infinitely many knots when $n\neq 0$, one may extend the definition of `good', \cite[Definition 3.14]{AJDLO_same_traces}, so that the $(+-)$ arc and $(-+)$ arc are lifted to have linking number $k \neq 0$ rel $D \times \{i, i+1\}$ (instead of just $k=\pm1$). This will still enable, almost verbatim, the calculation of the degree of the Alexander polynomial in \cite[Lemma 3.17]{AJDLO_same_traces} by which \cite[Theorem 3.13]{AJDLO_same_traces} distinguishes these infinite families of knots. The approach for $n=0$ remains the same.
\end{rem}

\begin{figure}
    \centering
    \includegraphics[width=0.2\linewidth]{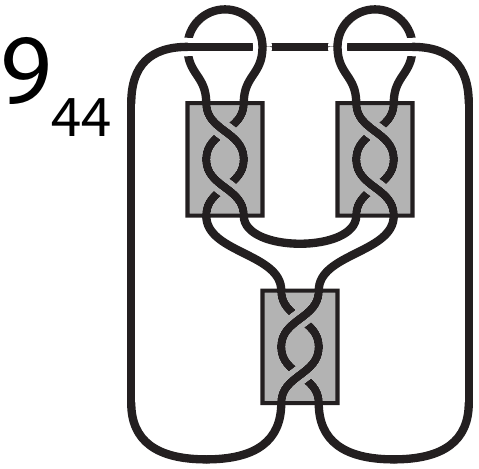}
    \caption{The knot $9_{44}$ is shown as a Montesinos knot. Its diagram suggests an infinite family of Montesinos knots that can be unknotted by a single positive crossing and by a single negative crossing.}
    \label{fig:9_44}
\end{figure}

\begin{rem}
In \cite[Definition 3.2]{AJDLO_same_traces} the $(*n)$ move is introduced, which is roughly a combination of an annulus twisting and $n$--twisting along some disk. It changes a knot $K$ to $K'$ whose $n$--traces are the same; see \cite[Theorem~3.7]{AJDLO_same_traces}.
\cite[Theorem 3.10]{AJDLO_same_traces} uses, for each integer $n$, repeated applications of the $(*n)$ move on a simple annulus presentation of a knot $K$ to produce infinitely many knots $K_n^k$ with the same $n$-trace as the knot $K$. See \cite{AJDLO_same_traces} for the specifics of the $(*n)$ move.  However, if the presentation is on a Hopf band with a right-handed twist (as in Figure~\ref{fig:GAP8_9}), these knots will all be isotopic when $n=-2$. This may be confirmed by surgery calculus on a surgery diagram of the knot $K_{-2}^1$ resulting from the $(*-2)$ move on the simple annulus presentation of $K=K_{-2}^0$. We do this in Figure~\ref{fig:n_minus2_general} with a description in the paragraph below. Consequently all the knots $K_{-2}^k$ for $k\in \Z$ obtained by repeated $(*-2)$ moves (and the inverse moves) are isotopic to $K$. Similarly, if the presentation is on a Hopf band with a left-handed twist, these knots will all be isotopic when $n=2$.

In the top row of Figure~\ref{fig:n_minus2_general}, (ii) shows a simple annulus presentation of a knot $K$, suppressing the specifics of the banding of the two boundary components of the Hopf band.  Note that for a simple annulus presentation, the banding is assumed to not intersect the interior of the inner region of this diagram. In (i), the twist of the Hopf band has been blown up, giving a surgery diagram of $K=K^0_{-2}$. From (ii) through (iii) to (iv) shows an isotopy of $K$. The isotopy from (iii) to (iv) slides the clasp along the band.  In the bottom row of Figure~\ref{fig:n_minus2_general}, (v) gives a surgery diagram of of $K^1_{-2}$, the result of applying the $(*-2)$ move on $K^0_{-2}$. To obtain (vi), we perform handle slides of the black knot over the red and the green knots as indicated, and then blow down the blue knot. Next we slide the black knot over the orange knot and then blow down the green knot to get (vii).  Finally we slide the black knot over the red knot and then cancel the red and the orange knots to yield (viii) which is equivalent to $K$ in (vi) above it.
\end{rem}

\begin{figure}
\centering
\includegraphics[width=0.99\textwidth]{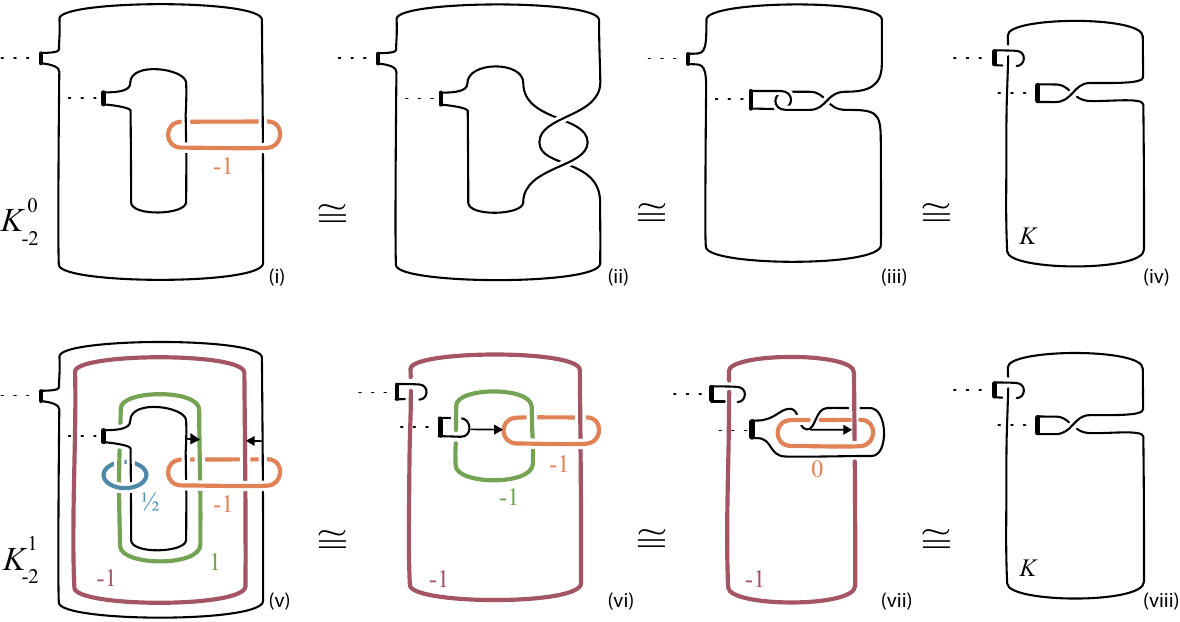}
\caption{
The knots $K^m_{-2}$ are all isotopic to $K$.} 
\label{fig:n_minus2_general}
\end{figure}

\section{Every concordance class contains a knot that is not trace-detected}
\label{sec:surgery_dual}

In this section, we prove Theorem~\ref{thm:concordance}, 
which we deduce from Theorem~\ref{thm:non-prime} below. 
The latter theorem, follows by reinterpreting a $3$-dimensional result from~\cite{Baker_Motegi_non_char} as a $4$-dimensional construction. 
Tagami~\cite[Remark 4.6]{Tagami_annulus} also described the methods in~\cite{Baker_Motegi_non_char} by using dualizable patterns \cite{dualizable} and RGB link constructions~\cite{RGB}. 

\begin{thm}
\label{thm:non-prime}
There exists a knot $K$ such that for any given non-trivial knot  $K'$,
$K \# K'$ is not detected by any of its traces. 
Furthermore, we have infinitely many distinct knots $K$ satisfying this property. 
\end{thm}

\begin{figure}[htbp]
\centering
\includegraphics[width=0.7\textwidth]{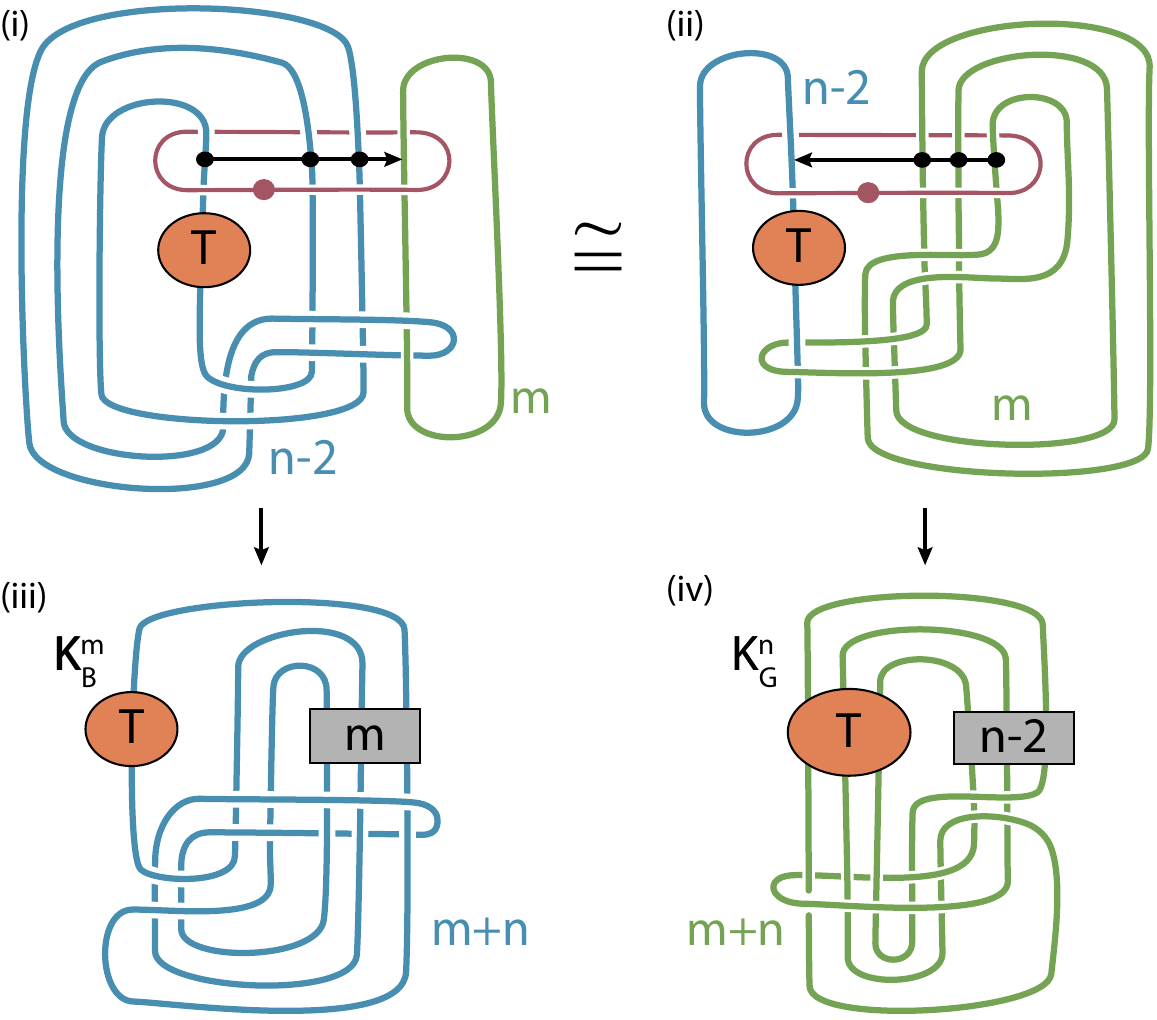} 
\caption{$K_B^m$ and $K_G^n$ share the same $(n+m)$-trace.  Grey boxes indicate number of full twists.
\label{RBG_BM_2}
}
\end{figure}

\begin{proof}
Figure~\ref{RBG_BM_2} (i) and (ii) show two isotopic Kirby diagrams with a $3$-component link $L = R \cup B \cup G$, 
where the shaded disk $T$ denotes a $1$-string tangle whose closure is the knot $K'$. 
In these diagrams, the blue knot $B$ has framing $n-2$ and the green knot $G$ has framing $m$.
The dotted red knot $R$, representing the $1$-handle, is a meridian of the blue knot $B$ and also of the green knot $G$, as visible in the two different presentations of $L$. 
Thus we can slide $B$ over $G$ three times successively along the dotted arrow and then cancel $G$ and $R$ to obtain the $(m+n)$-trace of a knot $K_B^m$, as shown in Figure~\ref{RBG_BM_2} (iii). 
By symmetry we can also slide $G$ over $B$ three times successively along the dotted arrow and then cancel $B$ and $R$ to obtain the $(m+n)$-trace of the knot $K_G^n$, as shown in Figure~\ref{RBG_BM_2} (iv). 
Since these two traces are also represented by isotopic Kirby diagrams, the knots $K_B^m$ and $K_G^n$ have the same $(m+n)$-trace. 

In Figure~\ref{RBG_BM_2} (iii) we see that the knot $K_B^m$ is the connected sum of $K'$ with a knot denoted by $K_m$. 
In~\cite[Example~4.5]{Baker_Motegi_non_char} it is shown that for every pair of integers $n$ and $m$ the knots $K_B^m$ and $K_G^n$ are non-isotopic. It follows that all knots $K_B^m=K_m\#K'$ are not detected by any trace. For concreteness we mention that $K_{0}=9_{42}$.
Furthermore, the Alexander polynomials of the knots $K_{m-1}$ computed in~\cite[Equation~$(\star)$]{Baker_Motegi_non_char} show that for $m\geq -1$ the knots $K_m$ are mutually distinct and thus we have infinitely many knots with the desired property. Note that the notation in~\cite{Baker_Motegi_non_char} differs from the notation used here, see Table~\ref{tab:comparison} for a comparison between these notions. 
\end{proof}

\begin{table}[htbp]
\caption{Translation of the notations used here and in~\cite{Baker_Motegi_non_char}.}
\begin{tabular}{lcccc} 
\toprule
Theorem~\ref{thm:non-prime} & $K_B^m$&$K_G^n$&$K_m$&$K'$\\
\cite[Example 4.5]{Baker_Motegi_non_char}&$k_{m+1}$&$K_{n-1}$&$k'_{m+1}$&$k''$\\
 \bottomrule
\end{tabular}
\label{tab:comparison}
\end{table}

\begin{rem}\hfill
\begin{enumerate}
    \item In Example 4.5 in~\cite{Baker_Motegi_non_char} it is also shown that $K_G^0$ is a prime satellite knot. Thus the proof of Theorem~\ref{thm:non-prime} also demonstrates the existence of a prime satellite knot that is not detected by any of its traces.
    \item The assumption in Theorem~\ref{thm:non-prime} that $K'$ is non-trivial is necessary for the construction used in this proof. Indeed, if $K'$ is the unknot, then $K_B^0$ and $K_G^0$ are isotopic. 
    On the other hand, if we start with the $3$-component link $L$ shown in Figure~\ref{fig:BMfig7-localknotRGB} we get, as in the proof of Theorem~\ref{thm:non-prime}, knots $K_B^m=K_m\#K'$ and $K_G^n$ that share the same $(n+m)$-trace. In this example, it turns out that for all pairs of integers $n$ and $m$, the knots $K_B^m$ and $K_G^n$ are non-isotopic. To distinguish $K_B^m$ and $K_G^n$ (also in the case when $K'$ is trivial) one can demonstrate that their HOMFLYPT polynomials are different. But since this calculation is lengthy and we do not need it to deduce Theorem~\ref{thm:concordance}, we do not add this here.
    \item We also get infinitely many mutually different hyperbolic knots that are not detected by any trace. This follows by combining~\cite{Tagami_annulus} and~\cite{Baker_Motegi_non_char}. Alternatively, this can be deduced from Figure~\ref{fig:9_44} as in the proof of Theorem~\ref{thm:main} or by putting a trivial tangle $T$ in Figure~\ref{fig:BMfig7-localknotRGB} and arguing as above.
\end{enumerate}
\end{rem}

\begin{figure}[htbp]
\centering
\includegraphics[width=0.6\textwidth]{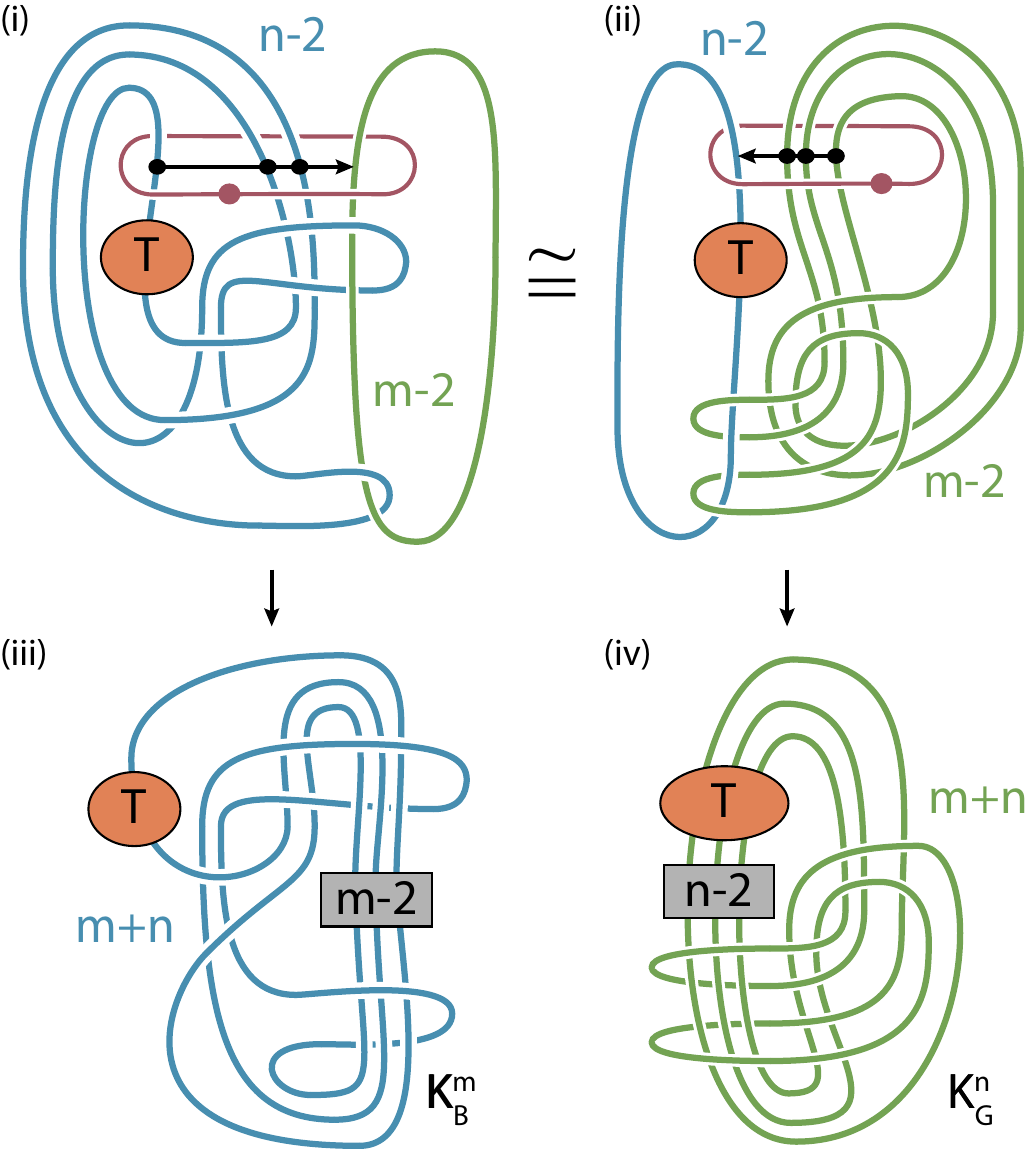}
\caption{$K_B^m$ and $K_G^n$ share the same $(n+m)$-trace.} 
\label{fig:BMfig7-localknotRGB}
\end{figure}

\begin{proof}[Proof of Theorem~\ref{thm:concordance}]
If $J$ is trivial, we replace $J$ by a non-trivial ribbon knot, which is necessarily ribbon concordant to the unknot.
Let $K$ be a knot with the property from Theorem~\ref{thm:non-prime}, for example we can take $K=9_{42}$.
We define knots $K'=(- \overline{K})\# J$ and $J'= K \#(- \overline{K})\# J$, where $- \overline{K}$ is the reflect inverse of the knot $K$.
Since $K\#(- \overline{K})$ is a ribbon knot, $J'$ is ribbon concordant to $J$.
By Theorem~\ref{thm:non-prime}, $J'=K\#K'$ is not detected by any of its traces. Moreover, Theorem~\ref{thm:non-prime} gives us infinitely many distinct knots $K$ with the above property and thus we get also infinitely many distinct knots $J'$ not detected by any of its traces, that are all concordant to $J$.  
\end{proof}

\subsection*{Individual grant support}

KLB was partially supported by the Simons Foundation gift \#962034.

MK is supported by the DFG, German Research Foundation, (Proj: 561898308); by a Ram\'on y Cajal grant (RYC2023-043251-I) and the project PID2024-157173NB-I00 funded by MCIN/AEI/10.13039/501100011033, by ESF+, and by FEDER, EU; and by a VII Plan Propio de Investigaci\'on y Transferencia (SOL2025-36103) of the University of Sevilla.

KM has been partially supported by JSPS KAKENHI Grant Number 25K07018, 21H04428, 23K03110, 23K20791 and Joint Research Grant of Institute of Natural Sciences at Nihon University for 2025.

\subsection*{Acknowledgments}

We thank John Baldwin and Steven Sivek for useful discussion and their interest in this article.
Also we thank the referee for careful reading and useful suggestions.

 \let\MRhref\undefined
 \bibliographystyle{hamsalpha}
 \bibliography{lit.bib}

\end{document}